%% file: collatz.tex
\documentclass[a4paper,11pt]{llncs}
\usepackage[utf8]{inputenc}
\usepackage[T1]{fontenc}
\usepackage{zi4}
\usepackage[a4paper,margin=2.5cm]{geometry}
\usepackage{cite}
\usepackage{mathtools,amsfonts}
\usepackage{mathptmx}
\usepackage[english]{babel}
\usepackage[final=true,babel,activate=true,kerning=true,protrusion=true]{microtype}
\usepackage{enumerate}
\usepackage{amssymb}
\usepackage{varioref}
\usepackage[pdfpagelabels=true]{hyperref}
\usepackage{cleveref}
\usepackage{xcolor}
\usepackage{xspace}
\usepackage{csquotes}
\usepackage[super]{nth}
\usepackage{url}

\usepackage[frozencache]{minted}

\hypersetup{
	linktoc = page,
	pdfpagemode = UseNone,
	colorlinks,
	linkcolor={red!50!black},
	citecolor={blue!50!black},
	urlcolor={blue!80!black}
}

\usepackage{booktabs}

\crefname{conjecture}{Conjecture}{Conjectures}

\usepackage{pgfplots}
\begin{document}

\title{A Fractional $3n+1$ Conjecture}

\author{%
    \'Eric Brier\inst{2} \and Rémi Géraud-Stewart\inst{1} \and David Naccache\inst{1}
} 

\institute{
    ÉNS (DI), Information Security Group,
CNRS, PSL Research University, 75005, Paris, France.\\
45 rue d'Ulm, 75230, Paris \textsc{cedex} 05, France\\
\email{\url{remi.geraud@ens.fr}},~~~~  
\email{\url{david.naccache@ens.fr}}
	\and
	\email{\url{eric.brier@polytechnique.org}}\\
}
\maketitle

\begin{abstract}
In this paper we introduce and discuss the sequence of \emph{real numbers} defined as $u_0 \in \mathbb R$ and $u_{n+1} = \Delta(u_n)$ where
\begin{equation*}
\Delta(x) = \begin{cases} \frac{x}{2} &\text{if } \operatorname{frac}(x)<\frac{1}{2} \\[4px] \frac{3x+1}{2} & \text{if } \operatorname{frac}(x)\geq\frac{1}{2} \end{cases}
\end{equation*}
This sequence is reminiscent of the famous Collatz sequence, and seems to exhibit an interesting behaviour. Indeed, we conjecture that iterating $\Delta$ will eventually either converge to zero, or loop over sequences of real numbers with integer parts $1,2,4,7,11,18,9,4,7,3,5,9,4,7,11,18,9,4,7,3,6,3,1,2,4,7,3,6,3$.

We prove this conjecture for $u_0 \in [0, 100]$. Extending the proof to larger fixed values seems to be a matter of computing power. The authors pledge to offer a reward to the first person who proves or refutes the conjecture completely --- with a proof published in a serious refereed mathematical conference or journal.
\end{abstract}

\section{Introduction}

Let, for all $x \in \mathbb R$, $\lfloor x\rfloor$ be the largest integer that is smaller or equal to $x$, and $\operatorname{frac} (x)=x - \lfloor x \rfloor$\footnote{These definitions hold for negative numbers as well, so that e.g. $\operatorname{frac}(-3/4) = 1/4$.}\textsuperscript{,}\footnote{Computer algebra systems may by default use another definition of the \enquote{fractional part} of a real number. Care is needed when implementing $\Delta$ with built-in functions.}. We then define the \emph{iteration function}:
\begin{equation*}
\Delta(x) = \begin{cases} \frac{x}{2} &\text{if } \operatorname{frac}(x)<\frac{1}{2} \\[4px] \frac{3x+1}{2} & \text{if } \operatorname{frac}(x)\geq\frac{1}{2} \end{cases}
\end{equation*}
Starting from any real $u_0$, the \emph{$\Delta$-sequence of seed $u_0$} is defined recursively as
\begin{align*}
    u_0, && u_1 & = \Delta(u_0), & u_2 & = \Delta(u_1) = \Delta^2(u_0), & \dotsc && u_n & = \Delta^n(u_0), \dotsc
\end{align*}
where $\Delta^n$ denotes $n$ successive applications of $\Delta$. This definition is reminiscent of the classical Collatz sequence, defined from the following iteration function\footnote{In this presentation, the division of $3x+1$ by $2$ is bundled together with it, as is common.} on the integers:
\begin{equation*}
    \textsf{Collatz}(x) = \begin{cases} \frac{x}{2} &\text{if } x \equiv 0 \pmod{2}\\[4px] \frac{3x+1}{2} & \text{if } x\equiv 1 \pmod{2} .\end{cases}
\end{equation*}
Note that $\Delta$-sequences never agree with a Collatz sequence on the integer (except trivially when the seed is $0$). A famously unsolved problem concerning the Collatz sequence is their limit cycle\footnote{Lagarias \cite{lagarias19853} likens the importance of this conjecture to the (still open) Catalan--Dickson conjecture on aliquot sequences or the (now proven) Last Theorem of Fermat.}:
\begin{conjecture}[Collatz 1937 \cite{col}]\label{conj:collatz}
    For any $x_0 \in \mathbb N_{>0}$, there exists $N \in \mathbb N$ such that $\operatorname{Collatz}^N(x_0) = 1$.
\end{conjecture}
An impressive body of mathematics has been produced to try and prove or disprove this claim. In 2019 Tao published a proof that almost all orbits of the Collatz map attain almost bounded values, which is perhaps the closest result to date to a complete proof of \Cref{conj:collatz} \cite{tao2019almost}.

Meanwhile, many natural generalisations of the Collatz sequence and conjecture have been proposed. Unfortunately they either happened to be (demonstrably) undecidable \cite{conway1972unpredictable}, harder than the original \cite{lagarias19853}, or trivial to solve. Our function $\Delta$, while only a variation rather than a generalisation, seems to avoid these two extremes: it seems to exhibit an interesting behaviour --- in particular a Collatz-like limit cycle --- in a context where more tools are available to study the resulting dynamical system. 
As such, study of the $\Delta$-sequences may provide new insights and tools that could, presumably, further our understanding of the Collatz sequence and other related problems. 

Computer simulation for many seed values led us to formulate the following:
\begin{conjecture}\label{conj:delta}
Iterating $\Delta$ will eventually either converge to zero, or loop over sequences of real numbers whose successive integer parts are:
\begin{equation*}
    1,2,4,7,11,18,9,4,7,3,5,9,4,7,11,18,9,4,7,3,6,3,1,2,4,7,3,6,3, \dotsc
\end{equation*}
This cycle has length 29.
\end{conjecture}
Similar to how Collatz sequences may reach large values before ultimately reaching the limit cycle, $\Delta$-sequences wander off but always seem to fall back to the conjectured situation. This is backed up by numerical simulation (see following example and \Cref{fig:bifurc}).

\begin{figure}[!hp]
    \centering
    \includegraphics[width=\textwidth]{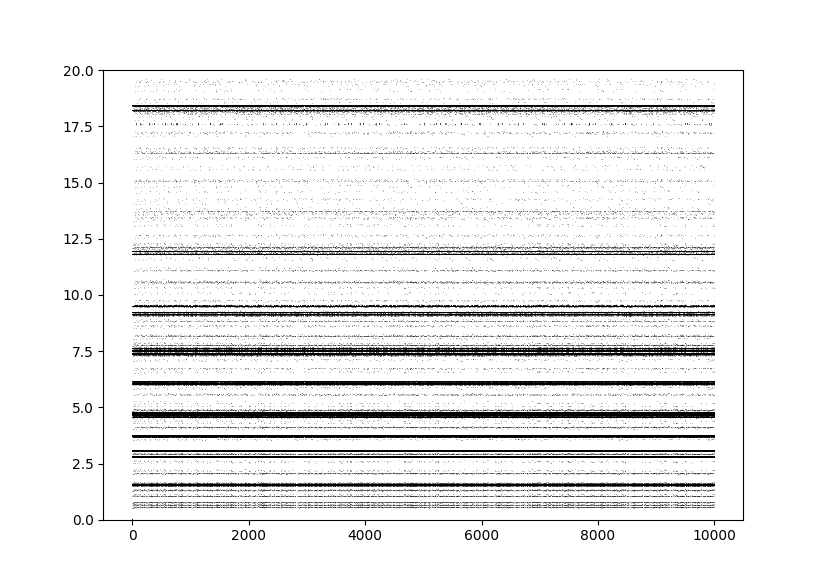}
    \caption{Values taken by the first $10^4$ items of $\Delta$-sequences as a function of the initial value $u_0$. The strongest lines correspond to values taken the most often, which are those in the conjectured cycle.}
    \label{fig:bifurc}
\end{figure}

\begin{example}[Numerical simulation] The following table shows successive iterations of $\Delta$, starting from $u_0 = 27$. At the \nth{28} iteration, the sequence enters the conjectured limit cycle (starting with $1.52...$).
\begin{center}
\texttt{\small{
\begin{tabular}{|c|c|c|c|c|c|c|c|c|c|c|c|c|c|c|c|}\hline
~start:~~&27.00& 13.50& 6.750& 10.63& 16.44& 8.219& 4.109& 2.055& 1.027& 0.514& 1.271& 0.635& 1.453& 0.726& 1.590\\
&2.884& 4.827& 7.740& 12.11& 6.055& 3.028& 1.514& 2.771& 4.656& 7.484& 3.742& 6.113& 3.056&      &      \\\hline
~cycle 1:~~&1.528& 2.792& 4.689& 7.533& 11.80& 18.20& 9.099& 4.550& 7.325& 3.662& 5.993& 9.490& 4.745& 7.618& 11.93\\ 
&18.39& 9.195& 4.597& 7.396& 3.698& 6.047& 3.024& 1.512& 2.768& 4.651& 7.477& 3.739& 6.108& 3.054&\\\hline
~cycle 2:~~&1.527& 2.790& 4.686& 7.529& 11.79& 18.19& 9.095& 4.547& 7.321& 3.660& 5.991& 9.486& 4.743& 7.615& 11.92\\
&18.38& 9.191& 4.596& 7.394& 3.697& 6.045& 3.023& 1.511& 2.767& 4.650& 7.476& 3.738& 6.107& 3.053&$\ldots$\\\hline
\end{tabular}}}
\end{center}
\end{example}
The central result of this paper is the following:
\begin{theorem}\label{thm:conj}
\Cref{conj:delta} is true for $0 \leq u_0 \leq 100$.
\end{theorem}

\section{Proof of \Cref{thm:conj}}
The proof strategy is as follows: we show that $\Delta$-sequences enter a uniquely-determined length-29 cycle when the seed belongs to an interval $I$. Then we extend this interval iteratively by showing that values in a larger interval $I_1$ eventually fall inside of $I$ when repeatedly applying $\Delta$. We repeat this process until the extended interval contains $[0, 100]$. 

Clearly the limit of 100 is arbitrary, and it is only a matter of computational power to obtain larger intervals. 

\subsection{The initial interval}
\begin{lemma} The number
\begin{equation*}
    x_0 = \frac{616136875}{407730749} = \frac{616136875}{2^{29} - 3^{17}}= 1.51113664228448956\ldots
\end{equation*}
is a fixed point of $\Delta^{29}$.
\end{lemma}

\begin{proof} Direct computation, see \Cref{app:lemma1}.
\qed 
\end{proof}

\begin{lemma}\label{lem:lemma2}
Over the interval $I=\left[\frac{3}{2},\frac{41}{27}\right[$,
\begin{equation*}
\Delta^{29}(x)=\frac{3^{17} x + 616136875}{2^{29}}.
\end{equation*}
\end{lemma}

\begin{proof}
Starting from any $x \in I$, we expand the computation graph at each iteration. Simplifying the expression yields the given formula, see \Cref{app:lemma2}.  
 \qed 
\end{proof}

\begin{lemma}\label{lem:lemma3}
The value $x_0$ is an \emph{attractive} fixed point of $\Delta^{29}$, i.e.,
\begin{equation*}
\forall x\in I, \quad \lim_{i\rightarrow{\infty}}\Delta^{29i}(x)=x_0.
\end{equation*}
\end{lemma}

\begin{proof}
Using \Cref{lem:lemma2}, we can easily compute the slope of $\Delta^{29}$, which equals $3^{17}/2^{29} < 1$. Since $x_0\in I$, successive iterations of $\Delta^{29}$ there converge to the fixed point. \qed 
\end{proof}
\begin{remark}
Note that \Cref{lem:lemma3} also determines the behaviour on $I$ of $\Delta^{29i+j}$ for $1<j<29$ when $i\to{\infty}$. 
\end{remark}

\subsection{Extending the interval}
Denote $I_0=I$. The strategy now consists in finding a new interval $I_1$, such that for every element $x\in I_1$ there exists an integer $\ell$ such that $\Delta^\ell(x) \in I$. In other words, $x \in I_1$ is eventually sent to $I$, where \Cref{lem:lemma3} applies, and therefore ultimately reaches the conjectured cycle.

We repeat this process on $I_1$, obtaining an interval $I_2$ such that every point in $I_2$ is eventually sent by $\Delta$ to $I_0\cup I_1$. Iterating, we therefore construct a sequence of increasingly larger intervals $J_k = I_0\cup\ldots\cup I_k$, in which the following invariant will hold true:
\begin{equation*}
    \forall x \in J_k,  \exists {\ell} \text{ s.t. } \Delta^{\ell}(x) \in I.
\end{equation*}
Therefore, by \Cref{lem:lemma3}, $x$ ultimately enters the conjectured cycle: proving \Cref{thm:conj} amounts to computing $J_0, J_1, \dotsc, J_k$ until $J_k \supseteq [0, 100]$. 

\begin{lemma}
\Cref{conj:delta} is true for $u_0 \in [\frac32, 20]$.
\end{lemma}

\begin{proof}
The proof consists in computing the backwards graph of $\Delta$, see \Cref{app:lemma4}. 
\qed 
\end{proof}
Two additional results are necessary to get the lower bound from $3/2$ down to zero. The first is the trivial observation that if $0 \leq x < 1/2$ then the sequence converges to $0$. The second is the following:
\begin{lemma} $\max_{\ell=0}^6 \Delta^\ell([\frac12, \frac32]) \subset [\frac32, 20]$.
\end{lemma}

\begin{proof}
Each function $\Delta^\ell$ is piecewise linear, therefore so is the maximum of these functions. It is then possible to explicitly compute the values of $\max \Delta$ over the interval, which are (strictly) contained inside $[\frac32, 20]$.

A plot of this function is given in \Cref{fig:lemma5}, see \Cref{app:lemma5}. The function is so tame that the graphical representation accurately represents all values.
\qed 
\end{proof}

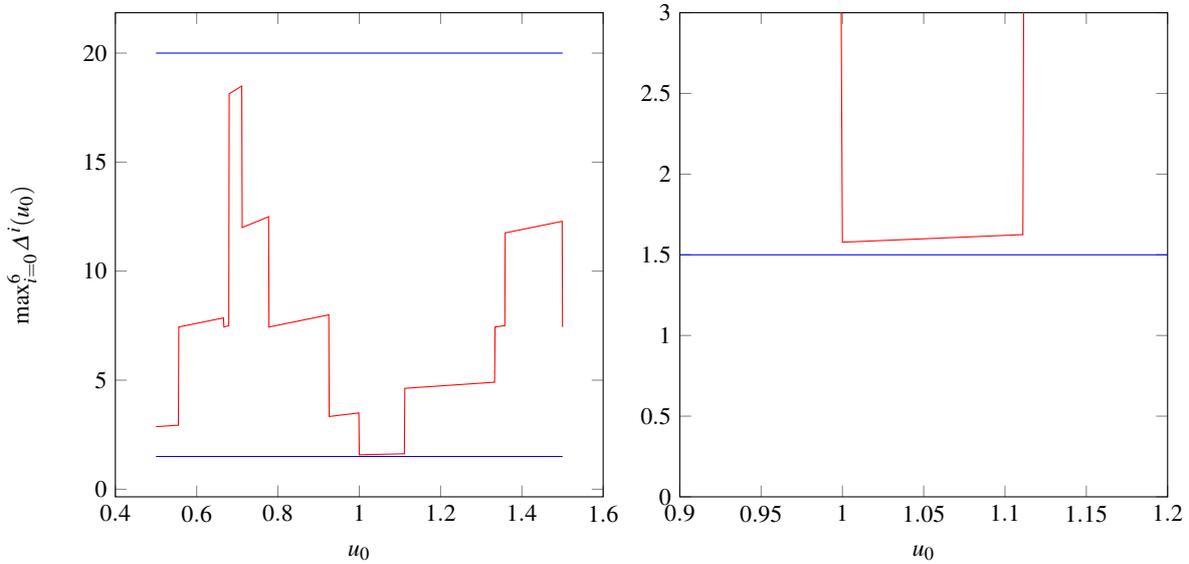
\begin{figure}[!ht]
    \centering
\input{img/delta6}
\caption{Graphical plot of $\max_{\ell=0}^6 \Delta^\ell$ over the interval $[1/2, 3/2]$ (left), and zoom on the interval $[0.9, 1.2]$ (right). The blue lines represents $y = 3/2$ and $y = 20$.}\label{fig:lemma5}
\end{figure}

Therefore, at the cost of up to 6 additional iterations, any value in the interval $[0, 20]$ either converges to zero (if $x < 1/2$) or enters the conjectured cycle. Reusing the program from \Cref{app:lemma4} on this extended interval, we can finally prove \Cref{thm:conj}:
\begin{proof}[of \Cref{thm:conj}]
The program
\begin{minted}[fontsize=\small]{mathematica}
Proof[1/2, 20, 101]
\end{minted}
reaches $100.77$ at the \nth{3626} iteration.\qed 
\end{proof}

\section{Open Problems}

\begin{description}
\item[Proving \Cref{conj:delta}.] Our approach turns a continuous problem into a discrete problem by exploiting the simple branching structure of $\Delta$. The interval extension algorithm is the bottleneck of this method, both theoretically and practically: indeed, the interval size grows as in \Cref{fig:plateaux}, exhibiting long plateaux (meaning no progress) of roughly similar length separated by abrupt climbs (meaning rapid progress). \Cref{conj:delta} is equivalent to the claim that the graph in \Cref{fig:plateaux} grows to infinity.

\smallskip

Improving the technique so that the plateaux are shorter would make the proof of \Cref{thm:conj} (and similar results) faster; eliminating them altogether would prove the conjecture.

\medskip

\item[Disproving \Cref{conj:delta}.]
Alternatively, one could look for a counterexample to \Cref{conj:delta}. As a result of \Cref{thm:conj}, such a value would have be at least larger than 100. It is unclear how to efficiently perform a search of this kind over the real line.

\smallskip

The authors offer \$61.6136875 (resp. €40.7730749) reward and a signed commemorative plaque to the first person who proves or disproves \cref{conj:delta}.

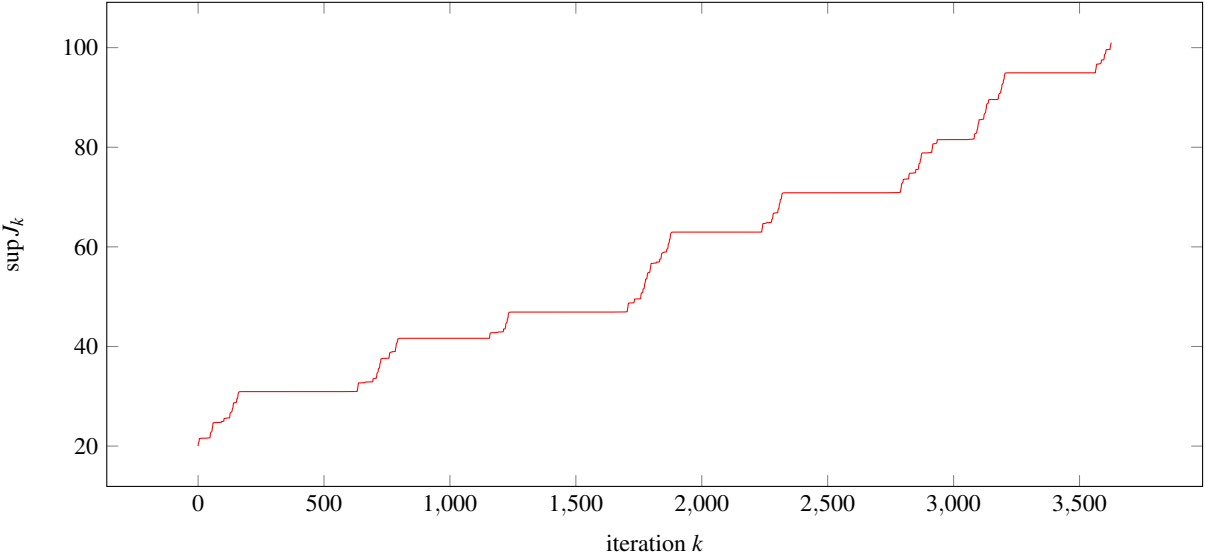
\begin{figure}[!hp]
    \centering
\input{img/plateaux}
\caption{Size of the extended interval $J_k$ as a function of $k$, in the proof of \Cref{thm:conj}.}\label{fig:plateaux}
\end{figure}

\medskip 

\item[Stopping time.] If \Cref{conj:delta} turns out to be true, results on the stopping time (number of iterations required to reach the cycle) are still to be found. Experimentally, it seems that almost all values smaller than 100 reach the cycle in fewer than 160 iterations, see \Cref{fig:stoppingtime,fig:stoppingtime2}. 

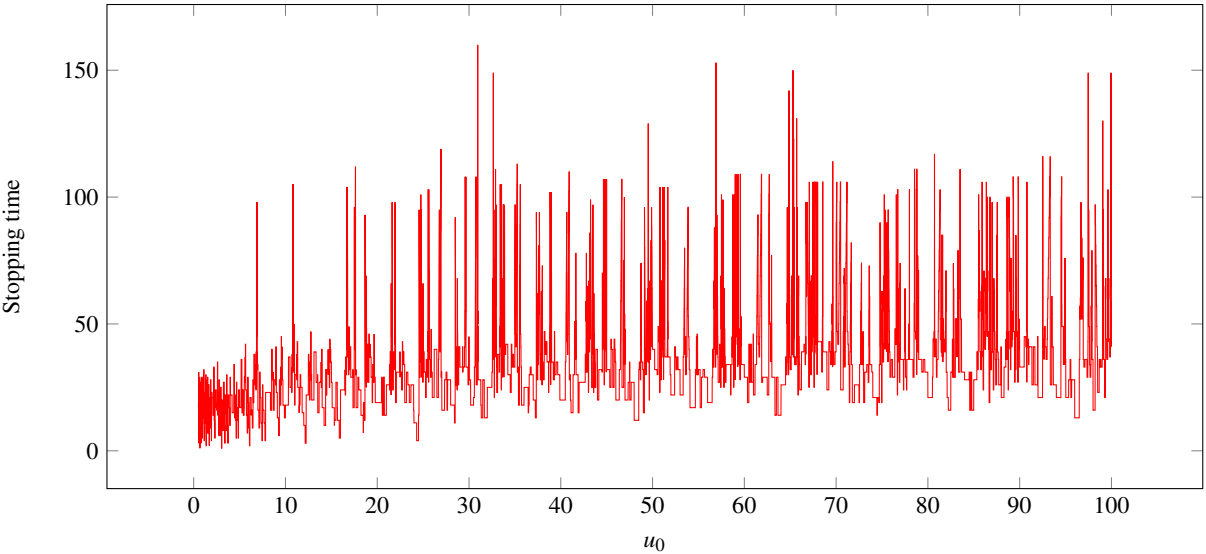
\begin{figure}[!h]
    \centering
    \input{img/tof}
    \caption{Stopping times for $u_0 \in [\frac12, 100]$.}
    \label{fig:stoppingtime}
\end{figure}

\begin{figure}[!h]
    \centering
    \input{img/tofhist}
    \caption{Histogram of stopping times for $u_0 \in [\frac12, 100]$.}
    \label{fig:stoppingtime2}
\end{figure}
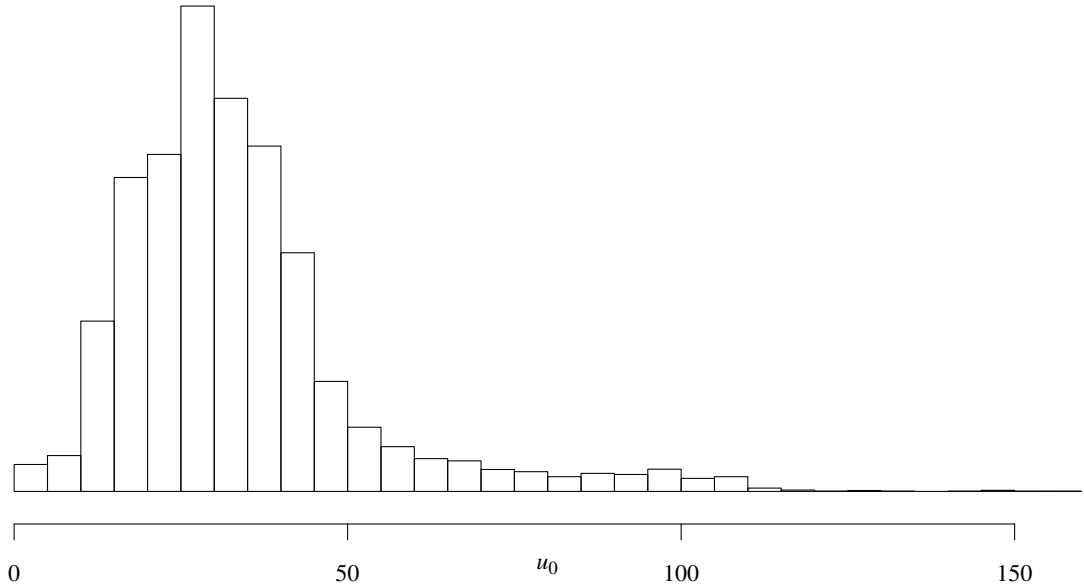

\medskip

\item[Maximal value.] Similarly, the maximum value reached by the sequence is an open question. Experimentally this maximum seems to be extremely sensitive about the initial conditions: for instance $u_0 = 30.94$ reaches $4662$, $u_0 = 30.944$ reaches $14956$, and $u_0 = 30.95$ reaches $59502$. 

\medskip 
\item[Relative phase.]
Given two seed values $(u, v) \in \mathbb R^2$, each producing a sequence that reaches the conjectured cycle, we can define the \emph{relative phase} of these numbers as the integer offset $\phi(u, v)$ between the two cycles --- since the time of flight to each sequence may vary substantially, we only consider the relative difference modulo 29. This makes it possible to provide a graphical representation of $\phi$ over $\mathbb R^2$ as in \Cref{fig:phase}. Rapid phase variations correlate with a high sensitivity to initial conditions.

\begin{figure}[!h]
    \centering
    \includegraphics[width=0.45\textwidth]{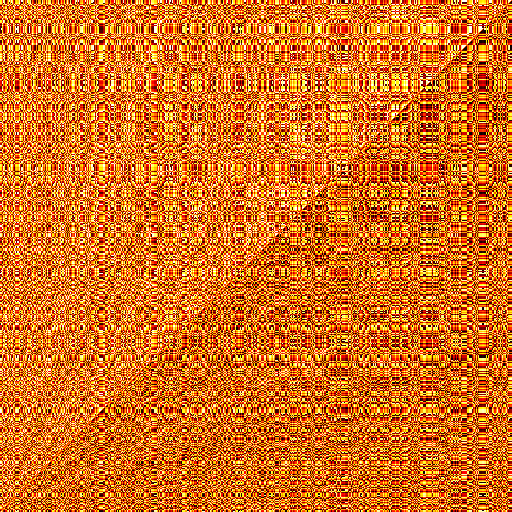}
    \includegraphics[width=0.45\textwidth]{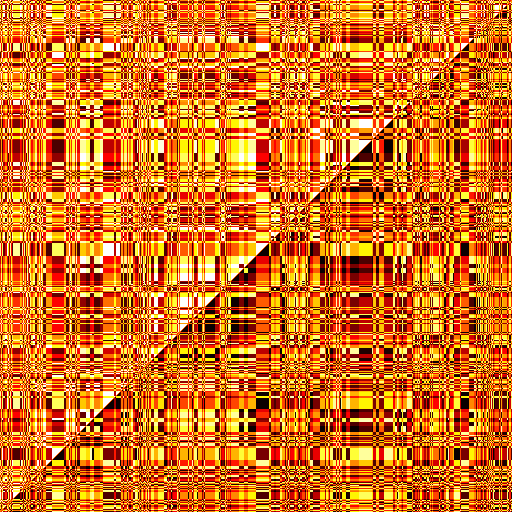}\\
    \includegraphics[width=0.45\textwidth]{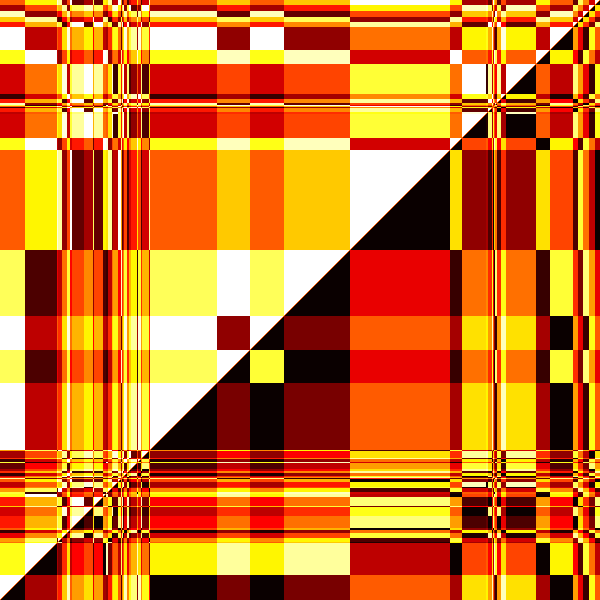}
    \caption{(Left) Relative phase for $(u, v) \in [\frac12, 100]^2$, corresponding to the region covered by \Cref{thm:conj}. Features appear there that could hint at structures holding clues about the sequence's behaviour. (Right) Relative phase in the region $(u,v) \in [70, 90]^2$. (Bottom) Relative phase in the region $(u,v) \in [\frac{318}{4}, \frac{325}{4}]^2$, showing the coexistence of rapidly varying phase and large areas of constant relative phase.}
    \label{fig:phase}
\end{figure}

\medskip

\item[Generalisations.] The coefficients in our definition of $\Delta$ are inspired by those of the Collatz function. Replacing $1/2$, $x/2$, or $(3x+1)/2$ in this definition by other quantities yields an infinite family of sequences whose properties remain to be investigated. For instance, is it possible to construct a variant of the $\Delta$ sequence that reaches a cycle of given, but arbitrary length?
\medskip 

\item[Extension to other Collatz variants.] The authors noted that when we replace in the original Collatz definition the ``if $x \equiv 0 \pmod{2}$'' case by  $\frac{x}{2}+c$ then we get convergence to different families of cycles. For instance for $x>2$ the choice $c=1$ causes convergence to one of 6 cycles of respective lengths 3, 4, 8, 8, 44, 44 whose minimal representatives are 7, 4, 21, 25, 189, 349. Convergence frequencies to those are very unbalanced with one cycle totaling $\simeq 50\%$ cases and another $\simeq 20\%$. For $c=3$ and $x\neq 6$ we get 10 cycles. $c=5$ and $x\neq3,10$ results in 2 cycles. $c=7$ and $x\neq 14$ yields 3 cycles etc. Studying fractional generalizations of those variants is yet another research target. For instance, the fractional sequence for $c=1$ \textsl{seems to} converge to the much shorter cycle $\frac{22}{5}\rightarrow\frac{16}{5}\rightarrow\frac{13}{5}$ which might be easier to analyze. Establishing general links between such fractional sequences and their integer versions may shed additional light on both our observations and Collatz-type sequences over $\mathbb{N}$.

\end{description}

\bibliographystyle{alpha}
\bibliography{coll.bib}

\appendix

\section{Mathematica code}
\subsection{Lemma 1}\label{app:lemma1}
The following code produces \texttt{True}:
\begin{minted}[fontsize=\small]{mathematica}
DeltaL[x_] := If[FractionalPart[x] < 0.5, x/2, (3x+1)/2];
x0 = 616136875 / 407730749;
Nest[DeltaL, x0, 29] == x0
\end{minted}

\subsection{Lemma 2}\label{app:lemma2}
The following Mathematica code proves that starting at any $x\in I$ we stitch at each iteration either an $x/2$ branch or a $(3x+1)/2$ branch:
\begin{minted}[fontsize=\small]{mathematica}
DeltaR[x_] := If[FractionalPart[x] <= 0.5, x/2, (3x+1)/2];
II[x_,y_]  := { DeltaL[x], DeltaR[y] };
r = N[ NestList[ II@@#&, {3/2, 41/27}, 29 ], 3 ]
\end{minted}
This code produces the following output:
\begin{center}
    
\texttt{\small{
\begin{tabular}{cccccc}
\{\{1.50, 1.52\},& \{2.75, 2.78\},& \{4.63, 4.67\},& \{7.44, 7.50\},& \{3.72, 3.75\},&\{6.08, 6.13\},\\
  \phantom{.}\{3.04, 3.06\},& \{1.52, 1.53\},& \{2.78, 2.80\},& \{4.67, 4.70\},&\{7.50, 7.54\},&\{11.8, 11.8\},\\ 
 \phantom{.}\{18.1, 18.2\},& \{9.07, 9.11\},& \{4.53, 4.56\},&\{7.30, 7.33\}, & \{3.65, 3.67\},& \{5.97, 6.00\},\\
 \phantom{.}\{9.46, 9.50\},& \{4.73, 4.75\},&\{7.60, 7.62\},& \{11.9, 11.9\},& \{18.3, 18.4\},& \{9.17, 9.20\},\\
 \phantom{.}\{4.59, 4.60\}, &\{7.38, 7.40\},& \{3.69, 3.70\},& \{6.03, 6.05\},& \{3.02, 3.03\},&\{1.51, 1.51\}\}\\
\end{tabular}}}
\end{center}
We can hence proceed and automatically compose:\smallskip
\begin{minted}[fontsize=\small]{mathematica}
For[ i=1, i <= 29,
     x[i+1] = If[ FractionalPart[r[[i,1]]] < 0.5, x[i]/2, (3x[i]+1)/2 ];
     i++];
\end{minted}
The further command \texttt{Expand[x[30]]} results in the printing of $\Delta^{29}$.

\subsection{Lemma 4}\label{app:lemma4}

Run the following Mathematica code:
\begin{minted}[fontsize=\small]{mathematica}
Proof := Function[{aIn, bIn, bOut},
  blst = {bIn}; i = 0;
  Print[Dynamic[{i, N[b, 100]}]];
  While[Max[blst] < bOut,
   b = blst[[-1]];
   x =.; t = b; f = x; m = Infinity;
   While[Not[b > t >= aIn],
    nm = Solve[f == Floor[2 t + 1]/2][[1, 1, 2]];
    m = Min[nm, m];
    f = If[FractionalPart[t] < 1/2, f, 3 f + 1]/2;
    t = DeltaL[t];
    ];
   nm = Solve[f == b, x][[1, 1, 2]];
   m = Min[nm, m];
   AppendTo[blst, m];
   i++]]
Proof[3/2, 41/27, 21]
\end{minted}
At iteration $227$ the bound $m$ reaches $20.77$.

\subsection{Lemma 5}\label{app:lemma5}
The following Mathematica code was used to generate \Cref{fig:lemma5}.

\begin{minted}[fontsize=\small]{mathematica}
Plot[{1.5, Max[NestList[DeltaL, x, 6]]}, {x, 0.5, 1.5}, PlotRange -> {1, 20}]
Plot[{1.5, Max[NestList[DeltaL, x, 6]]}, {x, 0.5, 1.5}, PlotRange -> {1, 2}]
\end{minted}

\section{Python/Numpy code}
The following Python/Numpy code can be used to produce the $\Delta$-sequence. It was used to generate the \enquote{tartan graphs} of \Cref{fig:phase} and the diagram of \Cref{fig:bifurc}.
\begin{minted}{python}
import numpy as np

def Delta(x):
    fracpart, _ = np.modf(x)
    return np.where(fracpart < 1/2, x/2, (3*x+1)/2)

# Run all iterations in parallel
x = np.linspace(0, 10)
for _ in range(1000):
    x = Delta(x)

# Each value is either 0 or in the cycle
print(x)
\end{minted}

\end{document}

%% file: img/delta6.tex
\begin{tikzpicture}
\begin{axis}[legend cell align = left,legend pos = north east,xlabel=$u_0$,ylabel=$\max_{i=0}^6 \Delta^i(u_0)$,width=0.5\textwidth,height=8cm]
\addplot[color=red, mark=.] table[x=x,y=y] {delta6.dat};
\addplot[color=blue, mark=.] table[x=x,y=min] {delta6.dat};
\addplot[color=blue, mark=.] table[x=x,y=max] {delta6.dat};
\end{axis}
\end{tikzpicture}
\begin{tikzpicture}
\begin{axis}[legend cell align = left,legend pos = north east,xlabel=$u_0$,width=0.5\textwidth,height=8cm,xmin=0.9,xmax=1.2,ymin=0,ymax=3]
\addplot[color=red, mark=.] table[x=x,y=y] {delta6.dat};
\addplot[color=blue, mark=.] table[x=x,y=min] {delta6.dat};
\addplot[color=blue, mark=.] table[x=x,y=max] {delta6.dat};
\end{axis}
\end{tikzpicture}

%% file: img/plateaux.tex
\begin{tikzpicture}
\begin{axis}[legend cell align = left,legend pos = north east,xlabel=$\text{iteration $k$}$,ylabel=$\sup J_k$,width=\textwidth,height=8cm]
\addplot[color=red, mark=.] table[x=x,y=tof] {plateaux.dat};
\end{axis}

\end{tikzpicture}

%% file: img/tof.tex
\begin{tikzpicture}
\begin{axis}[legend cell align = left,legend pos = north east,xlabel=$u_0$,ylabel=$\text{Stopping time}$,width=\textwidth,height=8cm]
\addplot[color=red, mark=.] table[x=x,y=tof] {ratios.dat};
\end{axis}

\end{tikzpicture}

%% file: img/tofhist.tex
\begin{tikzpicture}[x=1pt,y=1pt]
\begin{scope}
\definecolor{drawColor}{RGB}{0,0,0}

\node[text=drawColor,anchor=base,inner sep=0pt, outer sep=0pt, scale=  1.00] at (264.94, 15.60) {$u_0$};

\end{scope}
\begin{scope}[yshift=-1cm]
\definecolor{drawColor}{RGB}{0,0,0}

\path[draw=drawColor,line width= 0.4pt,line join=round,line cap=round] ( 65.18, 61.20) -- (439.74, 61.20);

\path[draw=drawColor,line width= 0.4pt,line join=round,line cap=round] ( 65.18, 61.20) -- ( 65.18, 55.20);

\path[draw=drawColor,line width= 0.4pt,line join=round,line cap=round] (190.03, 61.20) -- (190.03, 55.20);

\path[draw=drawColor,line width= 0.4pt,line join=round,line cap=round] (314.89, 61.20) -- (314.89, 55.20);

\path[draw=drawColor,line width= 0.4pt,line join=round,line cap=round] (439.74, 61.20) -- (439.74, 55.20);

\node[text=drawColor,anchor=base,inner sep=0pt, outer sep=0pt, scale=  1.00] at ( 65.18, 39.60) {0};

\node[text=drawColor,anchor=base,inner sep=0pt, outer sep=0pt, scale=  1.00] at (190.03, 39.60) {50};

\node[text=drawColor,anchor=base,inner sep=0pt, outer sep=0pt, scale=  1.00] at (314.89, 39.60) {100};

\node[text=drawColor,anchor=base,inner sep=0pt, outer sep=0pt, scale=  1.00] at (439.74, 39.60) {150};

\end{scope}
\begin{scope}[yshift=0.25cm,yscale=0.5]
\definecolor{drawColor}{RGB}{0,0,0}

\path[draw=drawColor,line width= 0.4pt,line join=round,line cap=round] ( 65.18, 75.85) rectangle ( 77.67, 96.18);

\path[draw=drawColor,line width= 0.4pt,line join=round,line cap=round] ( 77.67, 75.85) rectangle ( 90.15,102.89);

\path[draw=drawColor,line width= 0.4pt,line join=round,line cap=round] ( 90.15, 75.85) rectangle (102.64,204.34);

\path[draw=drawColor,line width= 0.4pt,line join=round,line cap=round] (102.64, 75.85) rectangle (115.12,312.71);

\path[draw=drawColor,line width= 0.4pt,line join=round,line cap=round] (115.12, 75.85) rectangle (127.61,330.11);

\path[draw=drawColor,line width= 0.4pt,line join=round,line cap=round] (127.61, 75.85) rectangle (140.09,442.04);

\path[draw=drawColor,line width= 0.4pt,line join=round,line cap=round] (140.09, 75.85) rectangle (152.58,372.45);

\path[draw=drawColor,line width= 0.4pt,line join=round,line cap=round] (152.58, 75.85) rectangle (165.06,336.40);

\path[draw=drawColor,line width= 0.4pt,line join=round,line cap=round] (165.06, 75.85) rectangle (177.55,255.91);

\path[draw=drawColor,line width= 0.4pt,line join=round,line cap=round] (177.55, 75.85) rectangle (190.03,158.85);

\path[draw=drawColor,line width= 0.4pt,line join=round,line cap=round] (190.03, 75.85) rectangle (202.52,124.27);

\path[draw=drawColor,line width= 0.4pt,line join=round,line cap=round] (202.52, 75.85) rectangle (215.00,109.60);

\path[draw=drawColor,line width= 0.4pt,line join=round,line cap=round] (215.00, 75.85) rectangle (227.49,100.58);

\path[draw=drawColor,line width= 0.4pt,line join=round,line cap=round] (227.49, 75.85) rectangle (239.97, 98.91);

\path[draw=drawColor,line width= 0.4pt,line join=round,line cap=round] (239.97, 75.85) rectangle (252.46, 92.41);

\path[draw=drawColor,line width= 0.4pt,line join=round,line cap=round] (252.46, 75.85) rectangle (264.94, 90.73);

\path[draw=drawColor,line width= 0.4pt,line join=round,line cap=round] (264.94, 75.85) rectangle (277.43, 86.96);

\path[draw=drawColor,line width= 0.4pt,line join=round,line cap=round] (277.43, 75.85) rectangle (289.92, 89.47);

\path[draw=drawColor,line width= 0.4pt,line join=round,line cap=round] (289.92, 75.85) rectangle (302.40, 88.63);

\path[draw=drawColor,line width= 0.4pt,line join=round,line cap=round] (302.40, 75.85) rectangle (314.89, 92.62);

\path[draw=drawColor,line width= 0.4pt,line join=round,line cap=round] (314.89, 75.85) rectangle (327.37, 85.70);

\path[draw=drawColor,line width= 0.4pt,line join=round,line cap=round] (327.37, 75.85) rectangle (339.86, 86.96);

\path[draw=drawColor,line width= 0.4pt,line join=round,line cap=round] (339.86, 75.85) rectangle (352.34, 78.36);

\path[draw=drawColor,line width= 0.4pt,line join=round,line cap=round] (352.34, 75.85) rectangle (364.83, 76.90);

\path[draw=drawColor,line width= 0.4pt,line join=round,line cap=round] (364.83, 75.85) rectangle (377.31, 76.06);

\path[draw=drawColor,line width= 0.4pt,line join=round,line cap=round] (377.31, 75.85) rectangle (389.80, 76.48);

\path[draw=drawColor,line width= 0.4pt,line join=round,line cap=round] (389.80, 75.85) rectangle (402.28, 76.06);

\path[draw=drawColor,line width= 0.4pt,line join=round,line cap=round] (402.28, 75.85) rectangle (414.77, 75.85);

\path[draw=drawColor,line width= 0.4pt,line join=round,line cap=round] (414.77, 75.85) rectangle (427.25, 76.06);

\path[draw=drawColor,line width= 0.4pt,line join=round,line cap=round] (427.25, 75.85) rectangle (439.74, 76.69);

\path[draw=drawColor,line width= 0.4pt,line join=round,line cap=round] (439.74, 75.85) rectangle (452.22, 76.06);

\path[draw=drawColor,line width= 0.4pt,line join=round,line cap=round] (452.22, 75.85) rectangle (464.71, 76.06);
\end{scope}
\end{tikzpicture}

%% file: collatz.bbl
\begin{thebibliography}{Con72}

\bibitem[Col86]{col}
Lothar Collatz.
\newblock {On the Motivation and Origin of the $(3n+ 1)$-Problem}.
\newblock {\em J. of Qufu Normal University}, 12(3):9--11, 1986.

\bibitem[Con72]{conway1972unpredictable}
John~H. Conway.
\newblock Unpredictable iterations.
\newblock 1972.

\bibitem[Lag85]{lagarias19853}
Jeffrey~C Lagarias.
\newblock {The $3 x+ 1$ Problem and its Generalizations}.
\newblock {\em The American Mathematical Monthly}, 92(1):3--23, 1985.

\bibitem[Tao19]{tao2019almost}
Terence Tao.
\newblock {Almost all Orbits of the Collatz Map Attain Almost Bounded Values}.
\newblock {\em arXiv preprint arXiv:1909.03562}, 2019.

\end{thebibliography}
